\documentclass[a4paper,11pt]{amsart}

\usepackage{graphics}

\usepackage[T1]{fontenc}    

\usepackage{amsthm}
\usepackage{amsbsy,amsmath,amssymb,amscd,amsfonts}
\usepackage[pagebackref=true]{hyperref}
\usepackage{url}            
\usepackage{booktabs}       
\usepackage{nicefrac}       
\usepackage{microtype}      

\usepackage{graphicx,float,latexsym,color}
\usepackage[font={scriptsize,it}]{caption}
\usepackage{subcaption}

\usepackage{makecell}

\usepackage[dvipsnames]{xcolor}

\newtheorem{theorem}{Theorem}

\newtheorem{proposition}{Proposition}
\newtheorem{conjecture}{Conjecture}
\newtheorem{remark}{Remark}

\hypersetup{
    pdftoolbar=true,        
    pdfmenubar=true,        
    pdffitwindow=false,     
    pdfstartview={FitH},    
    colorlinks=true,       
    linkcolor=OliveGreen,          
    citecolor=blue,        
    filecolor=black,      
    urlcolor=red           
}

\usepackage{lineno}

\title[]{Circum- and Inconic Invariants of\\3-Periodics in the Elliptic Billiard}

\author{Dan Reznik}
\address{Dan Reznik,
Data Science Consulting,\\
Rio de Janeiro, RJ, Brazil}
\email{dan@dat-sci.com}

\author{Ronaldo Garcia}
\address{Ronaldo Garcia,
Inst. de Matemática e Estatística,\\
Univ. Federal de Goiás,\\
Goiânia, GO, Brazil}
\email{ragarcia@ufg.br}

\begin{document}
\maketitle
\begin{abstract}
A Circumconic passes through a triangle's vertices; an Inconic is tangent to the sidelines. We study the variable geometry of certain conics derived from the 1d family of 3-periodics in the Elliptic Billiard. Some display intriguing invariances such as aspect ratio and pairwise ratio of focal lengths.

\end{abstract}

\section{Introduction}
\label{sec:intro}
Given a triangle, a {\em circumconic} passes through its three vertices and satifies two additional constraints, e.g., center or perspector\footnote{Where reference and polar triangles are perspective \cite{mw}.} location. An {\em Inconic} touches each side and is centered at a specified location. Both these objects are associated\footnote{The Isogonal and Isotomic conjugation of such conics are lines \cite[Perspector]{mw}.} with simple lines on the plane \cite[Circumconic,Inconic]{mw} and therefore lend themselves to agile algebraic manipulation.

We study properties and invariants of such conics derived from a 1d family of triangles: 3-periodics in an Elliptic Billiard (EB): these are triangles whose bisectors coincide with normals to the boundary (bounces are elastic), see Figure~\ref{fig:three-orbits-proof}.

\begin{figure}
    \centering
    \includegraphics[width=.66\textwidth]{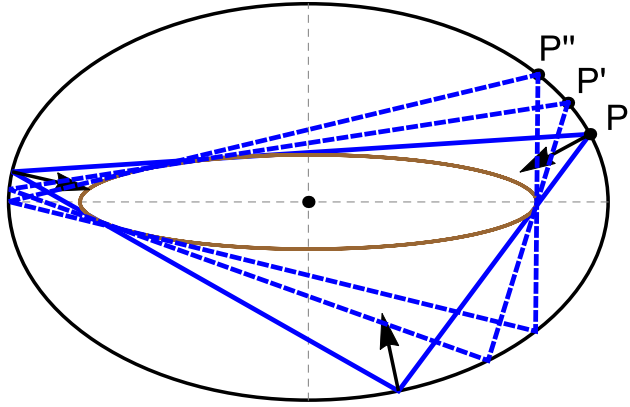}
    \caption{3-periodics (blue) in the Elliptic Billiard (EB, black): normals to the the boundary at vertices (black arrows) are bisectors. The family is constant-perimeter and envelopes a confocal Caustic (brown). This family conserves the ratio inradius-to-circumradius and has a stationary Mittenpunkt at the EB center. \textbf{Video}: \cite[PL\#01]{reznik2020-playlist-circum}.}
    \label{fig:three-orbits-proof}
\end{figure}

Amongst all planar curves, the EB is uniquely integrable \cite{kaloshin2018}. It can be regarded as a special case of Poncelet's Porism \cite{dragovic11}. These two propeties imply two classic invariances: $N$-periodics have constant perimeter and envelop a confocal Caustic. The seminal work is \cite{sergei91} and more recent treatments include \cite{lynch2019-billiards,rozikov2018}. 

We have shown 3-periodics also conserve the Inradius-to-Circumradius\footnote{As does the Poristic Family \cite{gallatly1914-geometry}.} ratio which implies an invariant sum of cosines, and that their {\em Mittenpunkt}\footnote{Where lines drawn from each Excenter thru sides' midpoints meet.} is stationary at the EB center \cite{reznik2020-intelligencer}. Indeed many such invariants have been effectively generalized for $N>3$ \cite{akopyan2020-invariants,bialy2020-invariants}.

We have also studied the loci of 3-periodic Triangle Centers over the family: out of the first 100 listed in \cite{etc}, 29 sweep out ellipses (a remarkable fact on its own) with the remainder sweeping out higher-order curves \cite{garcia2020-ellipses}. Related is the study of  loci described by the Triangle Centers of the Poristic Triangle family \cite{odehnal2011-poristic}. We have also showed \cite{reznik2020-circumbilliard} the aspect ratio of the {\em Circumbilliard} (a circumellipse which is an EB for a generic triangle) is invariant over the 1d family of {\em  Poristic}, triangles, with fixed Incircle and Circumcircle \cite{weaver1927-poristic,odehnal2011-poristic}. See \cite[PL\#07]{reznik2020-playlist-circum}.

\smallskip
\noindent \textbf{Main Results}:

\begin{itemize}
    \item Theorem~\ref{thm:axis-ratio} in Section~\ref{sec:circumellipses}: The ratio of semi-axis of the $X_1$-centered Circumellipse is invariant over the 3-periodic family. We conjecture this to be the case for a 1d-family of circumellipses.
    \item Theorem~\ref{thm:focal-ratio} in Section~\ref{sec:circumhyperbolae}: The focal lengths of two special circumhyperbola (Feuerbach and Excentral Jerabek) is constant over the 3-periodic family.
    \item Theorems~\ref{thm:excIncX5} and ~\ref{thm:excIncX3} and in Section~\ref{sec:inconic} show the aspect ratios of two important Excentral Inconics are invariant and that one of the Inconics is a $90^\circ$-rotated copy of the $X_1$-centered Circumellipse.
\end{itemize}

A reference table with all Triangle Centers, Lines, and Symbols appears in Appendix~\ref{app:symbols}. Videos of many of the experiments are assembled on Table~\ref{tab:playlist} in Section~\ref{sec:conclusion}.


\section{Invariants in Circumellipses}
\label{sec:circumellipses}
Let the boundary of the EB satisfy:

\begin{equation}
\label{eqn:billiard-f}
f(x,y)=\left(\frac{x}{a}\right)^2+\left(\frac{y}{b}\right)^2=1.
\end{equation}

Where $a>b>0$ denote the EB semi-axes throughout the paper. Below we use {\em aspect ratio} as the ratio of an ellipse's semi-axes. When referring to Triangle Centers we adopt Kimberling $X_i$ notation \cite{etc}, e.g., $X_1$ for the Incenter, $X_2$ for the Barycenter, etc., see Table~\ref{tab:kimberling} in Appendix~\ref{app:symbols}.

The following five-parameter equation is assumed for all circumconics not passing through $(0,0)$.

\begin{equation}
1 + c_1 x + c_2 y + c_3 x y + c_4 x^2 + c_5 y^2=0
\label{eqn:e0}
\end{equation}

\noindent In Appendix~\ref{app:circum-linear} we provide a method to compute the $c_i$ given 3 points a conic must pass through as well as it center.

The Medial Triangle divides the plane in 7 regions, see  Figure~\ref{fig:midlines}. The following is a known fact  \cite{akopyan2007-conics,odehnal2015-conics}:

\begin{remark}
If the center of a Circumconic lies within 4 of these (resp. the remainder 3), the conic will be an Ellipse (resp. Hyperbola).
\end{remark}

\begin{figure}
    \centering
    \includegraphics[width=\textwidth]{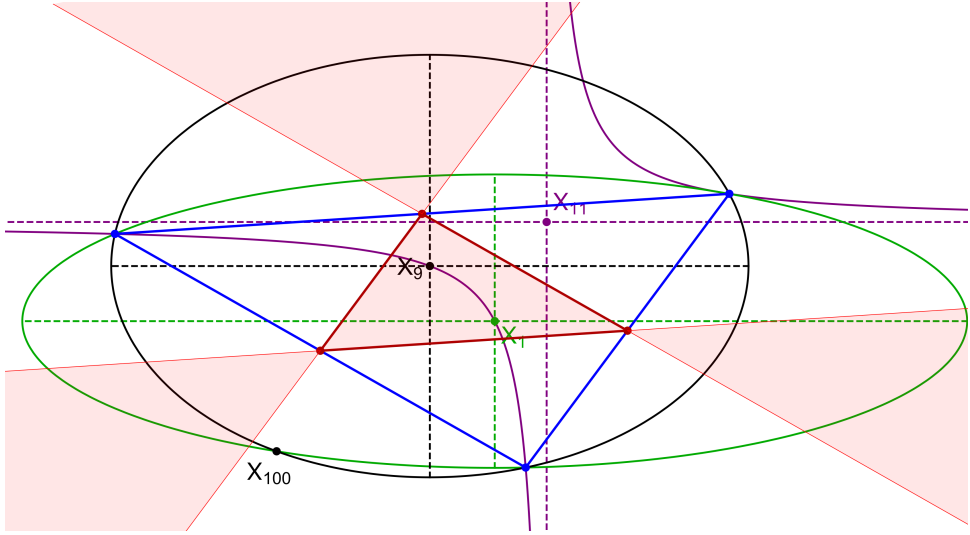}
    \caption{A reference triangle is shown (blue) as well as its Medial (red). The latter sides divide the plane into 7 regions, including the Medial' s interior. When a Circumconic center lies on any of the shaded regions (resp. unshaded) it is an Ellipse (resp. Hyperbola). Parabolas have centers at infinity. For illustration, the $X_1$ and $X_9$-centered Circumellipses and the $X_{11}$-centered Feuerbach Hyperbola are shown. Note that over the family of 3-periodics, a given Circumconic may alternate between Ellipse and Hyperbola, e.g., when centered on $X_4$, $X_5$, $X_6$, etc.}
    \label{fig:midlines}
\end{figure}

Centers $X_1$, $X_2$, and $X_9$ are always interior to the Medial Triangle \cite{etc}, so the Circumconics $E_i,i=1,2,9$ centered on them will ellipses, Figure~\ref{fig:circumX1X2}. $E_2$ is the {\em Steiner Circumellipse}, least-area over all possible Circumellipses \cite{mw}, and $E_9$ is the Elliptic Billiard.

It is known that $E_1$ intersects the EB and the Circumcircle at $X_{100}$, the Anticomplement of the Feuerbach Point. Also that $E_2$ intersects $E_9$ at $X_{190}$, the Yff Parabolic Point \cite{dekov14,etc}. These two ellipses intersect at $X_{664}$\footnote{This is the isogonal conjugate of $X_{663}$, i.e., $\mathcal{L}_{663}$ mentioned before is coincidentally its {\em Trilinear Polar} \cite{mw}.} \cite{moses2020-private-circumconic}.

Given a generic triangle $T$:

\begin{proposition}
The axes of $E_1$  are parallel to $E_9$'s. 
\end{proposition}

The proof is in Appendix~\ref{app:circum-x1x2x9}. 

\begin{theorem}
Over the 3-periodic family, the major $\eta_1'$ (resp. minor $\eta_1$) semiaxis lengths of $E_1$ are variable, though their ratio is invariant. These are given by:

\begin{align*}
\eta_1'=&R+d,\;\;\;\eta_1=R-d\\ 
\frac{\eta_1'}{\eta_1}
=&\frac{R+d}{R-d}=\frac{1+\sqrt{1-2\rho}}{\rho}-1 > 1
\end{align*}
\label{thm:axis-ratio}
\end{theorem}
\noindent Where $d=|X_1X_3|=\sqrt{R(R-2r)}$.


\begin{proof}
Calculate the ratio using vertex locations (see \cite{garcia2020-new-properties}) for an isosceles orbit, and then verify with a Computer Algebra System (CAS) the expression holds over the entire family. A related proof appears in \cite{reznik2020-poristic}.
\end{proof}

Note: experimentally $\eta_1'$ is maximal (resp. minimal) when the 3-periodic is an isosceles with axis of symmetry parallel to the EB's minor (resp. major) axis.

\begin{proposition}
The axes of $E_2$ are only parallel to $E_9$ if $T$ is isosceles.
\end{proposition}

See Appendix~\ref{app:circum-x1x2x9}.



\begin{figure}
    \centering
    \includegraphics[width=\textwidth]{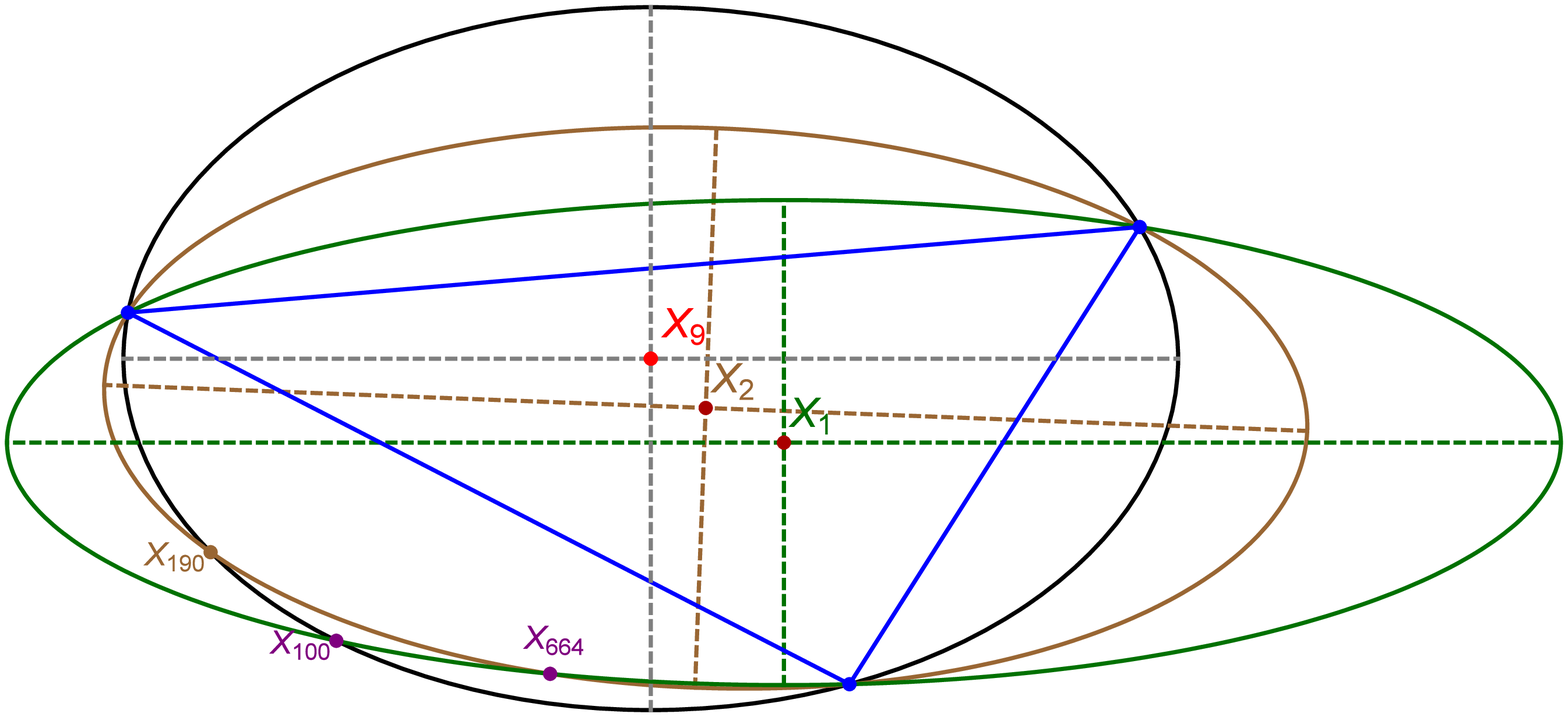}
    \caption{$E_1$ (green) and $E_2$ (brown) are Circumellipses to the orbit (blue) centered on $X_1$ (green) and $X_2$ (brown). The former (resp.~latter) intersect the EB at $X_{100}$ (resp.~$X_{190}$). They intersect at the vertices and at $X_{664}$. $E_1$ axes remain parallel to the EB over the orbit family, and the ratio of their lengths is constant. The axes of $E_2$ are only parallel to the EB's when the orbit is an isosceles. \textbf{Video}: \cite[PL\#08]{reznik2020-playlist-circum}.}
    \label{fig:circumX1X2}
\end{figure}

\subsection{Parallel-Axis Pencil}

The Feuerbach Circumhyperbola of a Triangle is a rectangular hyperbola\footnote{Since it passes through the Orthocenter $X_4$ \cite{mw}.} centered on $X_{11}$ \cite{mw}.
Peter Moses has contributed a stronger generalization \cite{moses2020-private-circumconic}:

\begin{remark}
The pencil of Circumconics whose centers lie on the Feuerbach Circumhyperbola $F_{med}$ of the Medial Triangle have mutually-parallel axes. 
\end{remark}

The complement\footnote{The 2:1 reflection of a point about $X_2$.} of $X_{11}$ is $X_{3035}$ \cite{etc} so $F_{med}$ is centered there, see Figure~\ref{fig:circum_parallel}. The following is a list of Circumellipses whose centers lie on $F_{med}$ \cite{moses2020-private-circumconic}: $X_i$,  $i$=1, 3, 9, 10\footnote{Notice $X_{10}$ is the Incenter of the Medial. Interestingly, $X_8$, the Incenter of the ACT, does not belong to this select group.}, 119, 142, 214, 442, 600, 1145, 2092, 3126, 3307, 3308, 3647, 5507, 6184, 6260, 6594, 6600, 10427, 10472, 11517, 11530, 12631, 12639, 12640, 12864, 13089, 15346, 15347, 15348, 17057, 17060, 18258, 18642, 19557, 19584, 22754, 34261, 35204.

\begin{proposition}
\label{th:9}
A circumellipse has  center on $F_{med}$ iff it passes through $X_{100}$.
\end{proposition}

A proof appears in Appendix~\ref{app:ce_parallel}. The following has been observed experimentally:

\begin{conjecture}
Over the family of 3-periodics, all Circumellipses in Moses' pencil conserve the ratio of their axes.
\label{conj:moses}
\end{conjecture}

\begin{figure}
    \centering
    \includegraphics[width=\textwidth]{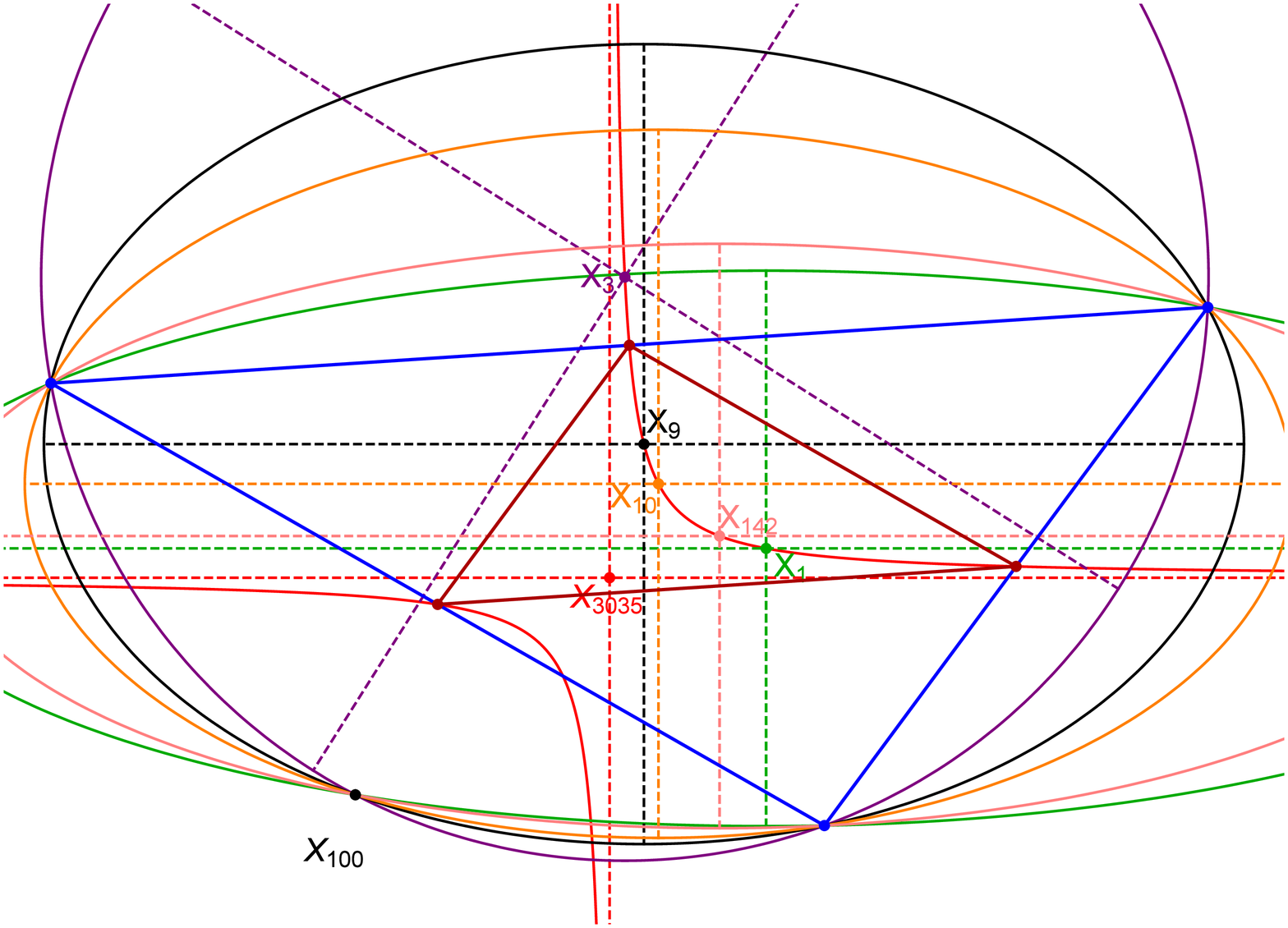}
    \caption{An $a/b=1.5$ EB is shown (black) centered on $X_9$ as well as a sample 3-periodic (blue). Also shown are Circumellipses centered on $X_i$, $i=$1, 3, 10, 142, whose centers lie on the Feuerbach Circumhyperbola of the Medial Triangle (both shown red), centered on $X_{3035}$, the complement of $X_{11}$. Notice all conics drawn (including the Circumhyp). have axes parallel to the EB and all Circumellipses pass through $X_{100}$. Note: the Circumellipse centered on $X_3$ is the Circumcircle, its axes, drawn diagonally, are immaterial.}
    \label{fig:circum_parallel}
\end{figure}

\section{A Special Pair of  Circumhyperbolae}
\label{sec:circumhyperbolae}

Here we study invariants of two well-known Circumhyperbolae\footnote{Its centers lie in the unshaded regions in Figure~\ref{fig:midlines}.}: the Feuerbach and Jerabek Hyperbolas $F$ and $J$ \cite[Jerabek Hyperbola]{mw}. Both are rectangular since they contain $X_4$ \cite{mw}. The former is centered on $X_{11}$ and the latter on $X_{125}$. With respect to 3-periodics no invariants have been detected for $J$. However, the Jerabek $J_{exc}$ of the Excentral Triangle, which passes through the Excenters and is centered on $X_{100}$\footnote{The Excentral's $X_{125}$ \cite{mw}.}, does produce interesting invariants\footnote{The Feuerbach Hyperbola $F_{exc}$ has not yet yielded any detectable invariants over the 3-periodic family.}. 
$F$ is known to pass through $X_1$ and $X_9$ of its reference triangle. Interestingly
$J_{exc}$ also passes through $X_1$ and $X_9$. This stems from the fact that $J$ passes through $X_4$ and $X_6$. Since the Excentral Triangle is always acute \cite{coxeter67}, its $X_4$ is $X_1$. Likewise, the excentral $X_6$ is $X_9$. 

The Isogonal Conjugate of a Circumconic is a line \cite[Circumconic]{mw}. Remarkably:

\begin{remark}
The Isogonal conjugate of $F$ with respect to a reference triangle and that of $J_{exc}$ with respect to the Excentral one is line $X_1X_3=\mathcal{L}_{650}$.
\end{remark}

The first part is well-known \cite[Feuerbach Hyperbola]{mw}. For the second part, consider that $J$ is the Isogonal Conjugate of the Euler Line \cite[Jerabek Hyperbola]{mw}. The Euler Line of the Excentral Triangle passes through its $X_4$ and $X_5$ which are $X_1$ and $X_3$ in the reference 3-periodic.

Referring to Figure~\ref{fig:circumhyps}:

\begin{proposition}
$J_{exc}$ intersects $E_9$ in exactly two locations.
\end{proposition}

\begin{proof}
Let $s_i,i=1,2,3$ refer to 3-periodic sidelengths. The perspector of $J_{exc}$ is $X_{649}=s_1(s_2-s_3)::$ (cyclical) \cite{etc}. Therefore the trilinears $x:y:z$ of $J_{exc}$ satisfy \cite{yiu2003}:
 \[J_{exc}: {s_1(s_2-s_3)}x^2+s_2(s_3-s_1)y^2+s_3(s_1-s_2)z^2=0.\]
 
Notice the above is satisfied for the Excenters $[1:1:-1],\; [1:-1:1]$ and $ [-1:1:1]$. As $X_1=[1:1:1]$ and $X_9=s_2+s_3-s_1::$ (cyclical) it follows that
 $J_{exc}(X_1)=J_{exc}(X_9)=0$.
 
 Eliminating variable $x$, the intersection of $J_{exc}=0$ and $E_9=0$ is given by the quartic:
 
 \begin{align*}
 &s_2(s_1-s_3)k_1 y^4+2s_2(s_1-s_3){k_1}{k_2}y^3 z\\
 +&2s_3(s_1-s_2){k_1}{k_2}y z^3+s_3(s_1-s_2)k_1 z^4=0
 \end{align*}
 
With $k_1=(s_1+s_2-s_3)^2$ and $k_2=s_1+s_3-s_2$. The discriminant of the above equation is:
 
 \[
 -432[(s_2-s_3)(s_1-s_3)(s_1-s_2)(s_1+s_3-s_2)^2(s_1-s_2-s_3)^2(s_1+s_2-s_3)^2(s_1s_2s_3)]^2
 \]
 
Since it is negative, there will be two real and two complex solutions \cite{burnside_1960}.
\end{proof}

\begin{proposition}
$F$ intersects the $X_9$-centered Circumellipse at $X_{1156}$.
\end{proposition}

\begin{proof}
The perspector of $X_9$ is $X_1$ and that of $X_{11} $ is $X_{650}=(s_3-s_3)(s_3+s_3-s_1)::$ (cyclical). Therefore, the trilinears $x:y:z$ of $F$ and $E_9$ satisfy:

\begin{align*}
    F:&(s_2-s_3)(s_2+s_3-s_1)/x+\\
    &(s_3-s_1)(s_3+s_1-s_2)/y+\\
    &(s_1-s_2)(s_1+s_2-s_3)/z=0 \\
    E_9:& 1/x+1/y+1/z=0.
\end{align*} 

$X_{1156}$ is given by $1/[(s_2-s_3)^{2}+s_1( s_2+s_3-2s_1)]::$ (cyclic). This point can be readily checked to satisfy both of the above.
\end{proof}
 

Given a generic triangle $T$, the following two claims are known:

\begin{proposition}
The asymptotes of both $F$ and $J_{exc}$ are parallel to the $X_9$-centered circumconic, i.e., $c_4$ and $c_5$ in \eqref{eqn:e0} vanish.
\end{proposition}

\begin{proof}
To see the first part, consider that since the Caustic is centered on $X_9$ and tangent to the 3-periodics, it is the (stationary) Mandart Inellipse $I_9$ of the family \cite{mw}. This inconic is known to have axes parallel to the asymptotes of $F$ \cite{gibert2004-mandart}. Since the Caustic is confocal with the EB, $F$ asymptotes must be parallel to the EB axes.

Secondly, 3-periodics are the Orthic Triangles of the Excentrals, therefore the EB is the (stationary) Excentral's Orthic Inconic \cite{mw}. The latter's axes are known to be parallel to the asymptotes of the Jerabek hyperbola. \cite[Orthic Inconic]{mw}.

An alternate, algebraic proof appears in Appendix \ref{app:circum-x1x2x9}.
\end{proof}

\noindent Let $\lambda$ (resp. $\lambda'$) be the focal length of $F$ (resp. $J_{exc}$).

\begin{remark}
Isosceles 3-periodics have $\lambda'=\lambda=0$.
\end{remark}

To see this consider the sideways isosceles 3-periodic with $P_1=(a,0)$. $P_2$ and $P_3$ will lie on the 2nd and 3rd quadrants at $(-a_c,{\pm}y')$, where $a_c=a(\delta-b^2)/(a^2-b^2)$ is the length of the Caustic major semi-axis \cite{garcia2020-ellipses}. $X_1$ and $X_4$ will lie along the 3-periodic's axis of symmetry, i.e., the x-axis. To pass through all 5 points, $F$ degenerates to a pair of orthogonal lines: the x-axis and the vertical line $x=-a_c$. The foci will collapse to the point $(-a_c,0)$. A similar degeneracy occurs for the upright isosceles, i.e., when $P_1=(0,b)$, namely, the foci collapse to $(0,-b_c)$, where $b_c=b(a^2-\delta)/(a^2-b^2)$ is the Caustic minor semi-axis length.

\begin{theorem}
For all non-isosceles 3-periodics, $\lambda'/\lambda$ is invariant and given by:
\label{thm:focal-ratio}
\end{theorem}
 
\begin{equation}
 \frac{\lambda'}{\lambda}=
 \sqrt{2/\rho}>2
 \label{eqn:focal-ratio}
\end{equation}

\begin{proof}
Assume the EB is in the form of \eqref{eqn:billiard-f}. Let the 3-periodic be given by $P_i=(x_i,y_i),i=1,2,3$. $F$ passes through the $P_i$, $X_1$ and $X_9=(0,0)$. The asymptotes of $F$ are parallel to the EB axes, therefore this hyperbola is given by $c_1x+c_2y+c_3 x y = 0$ and $\lambda^2=|8c_1c_2/c_3^2|$, where: 
\begin{align}
c_1=& y_2y_3(x_2-x_3)x_1^2+(x_2^2y_3-x_3^2y_2-y_2^2y_3
+y_2y_3^2)x_1y_1 \nonumber\\
+&y_2y_3(x_2-x_3)y_1^2  
-(x_2y_3-x_3y_2)(x_2x_3+y_2y_3)y_1
 \nonumber \\
 c_2=&x_2x_3(y_2-y_3)x_1^2+(x_2x_3^2-x_2^2x_3-x_2y_3^2+x_3y_2^2)x_1y_1 \label{eqn:ci} \\
 +&(x_2y_3-x_3y_2)(x_2x_3+y_2y_3)x_1-x_2x_3(y_2-y_3)y_1^2 \nonumber\\
 c_3=& (x_2y_3-x_3y_2)x_1^2+
(x_3^2y_2-x_2^2y_3+  y_2^2y_3-y_2y_3^2)x_1 \nonumber\\
+&(x_3y_2-x_2y_3 )y_1^2 
+(x_2^2x_3-x_2x_3^2+x_2y_3^2-x_3y_2^2)y_1 \nonumber
\end{align}
 
Let $P_i'=(x_i',y_i'),i=1,2,3$ be the Excenters. They are given by
\begin{align} 
    P_1'=&\left({\frac {-{x_1}\,{s_1}+{x_2}\,{s_2}+{x_3}\,{s_3}}{{
s_2}+{s_3}-{s_1}}},{\frac {-{y_1}\,{s_1}+{y_2}\,{
s_2}+{y_3}\,{s_3}}{{s_2}+{s_3}-{s_1}}}\right) \nonumber\\
P_2'=& \left({\frac {{x_1}\,{s_1}-{x_2}\,{s_2}+{x_3}\,{s_3}}{{
s_3}+{s_1}-{s_2}}},{\frac {{y_1}\,{s_1}-{y_2}\,{
s_2}+{y_3}\,{s_3}}{{s_3}+{s_1}-{s_2}}}\right) \label{eqn:pi-prime} \\
P_3'=& \left({\frac {{x_1}\,{s_1}+{x_2}\,{s_2}-{x_3}\,{s_3}}{{
s_1}+{s_2}-{s_3}}},{\frac {{y_1}\,{s_1}+{y_2}\,{
s_2}-{y_3}\,{s_3}}{{s_1}+{s_2}-{s_3}}}\right) \nonumber
\end{align}
Here, $s_1=|P_2-P_3|$, $s_2=|P_1-P_3|$ and $s_3=|P_1-P_2|$.

Since $J_{exc}$ is also centered on the origin and has horizontal/vertical asymptotes, $J_{exc}$ is given by $c_1'x+c_2'y+c_3'x y=0$, and $(\lambda')^2=|8c_1'c_2'/(c_3')^2|$, where $c_i'$ are constructed as \eqref{eqn:ci} replacing $(x_i,y_i)$ with $(x_i',y_i')$.

Consider a right-triangle\footnote{We found this to best simplify the algebra.} 3-periodic, e.g., with $P_1(x^\perp,y^\perp)$ \cite[Proposition 5]{reznik2020-circumbilliard}:

\begin{equation*}
x^\perp=\frac{a^2 \sqrt{a^4+3 b^4-4 b^2 \delta }}{c^3},\;\;\;
y^\perp=\frac{b^2 \sqrt{-b^4-3 a^4+4 a^2 \delta }}{c^3}
\label{eqn:perp}
\end{equation*}

\noindent $P_2$ and $P_3$ are obtained explicitly \cite{garcia2019-incenter}. From these obtain $c_i$ using \eqref{eqn:ci}. Using \eqref{eqn:pi-prime} obtain $P_i'$ and the $c_i'$. Finally, obtain a symbolic expression for $\lambda'/\lambda$. After some manipulation and simplification with a Computer Algebra System (CAS), we obtain \eqref{eqn:focal-ratio} which we call a {\em candidate}.

Parametrize the 3-periodic family with $P_1(t)=(a\cos{t},b\sin{t})$ and using the sequence above arrive at an expression for $\lambda'/\lambda$ in terms of $t$. Subtract that from the right-triangle candidate. After some algebraic manipulation and CAS simplification verify the subtraction vanishes, i.e., $\lambda'/\lambda$ is independent of $t$.
\end{proof}

\begin{figure}
    \centering
    \includegraphics[width=\textwidth]{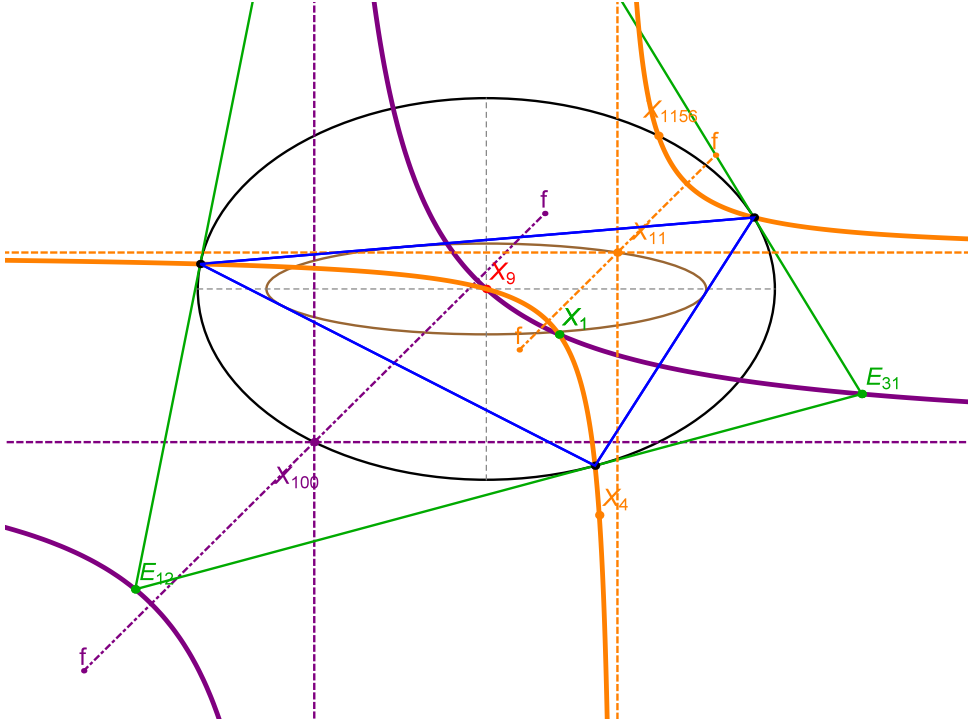}
    \caption{An $a/b=1.5$ EB is shown (black) as well as a sample 3-periodic (blue), the confocal Caustic (brown), and the Excentral Triangle (green). The 3-periodic's Feuerbach Circumhyperbola $F$ (orange) passes through its three vertices as well as $X_1$, $X_9$, and $X_4$. The Excentral's Jerabek Circumhyperbola $J_{exc}$ (purple) passes through the three Excenters, as well as $X_1$, $X_9$ and $X_{40}$ (not shown). Two invariants have been detected over the orbit family: (i) the asymptotes (dashed) of both $F$ and $J_{exc}$ stay parallel to the EB axes, (ii) the ratio of focal lengths is constant (focal axis appears dashed). $F$ intersects the Billiard at $X_{1156}$.}
    \label{fig:circumhyps}
\end{figure}

\subsection{Focal Length Extrema}

Let $P_1(t)=(a\cos{t},b\sin{t})$. While their ratio is constant, $\lambda$ and $\lambda'$ undergo three simultaneous maxima in $t\in(0,\pi/2)$, see Figure~\ref{fig:focal-ratio}. In fact, the following additional properties occur at configurations of maximal focal length (we omit the rather long algebraic proofs), see Figure~\ref{fig:zero-hyps}:

\begin{itemize}
    \item $F'$ is tangent to the Caustic at ${\pm}X_{11}$.
    \item $J'_{exc}$ is tangent to the EB at ${\pm}X_{100}$, i.e., at ${\mp}X_{1156}$ (see below).
\end{itemize}

\begin{remark}
Like $F$, $J'_{exc}$ intersects the EB at $X_{1156}$.
\end{remark}

This happens because $X_{1156}$ is the reflection of $X_{100}$ about $X_9$. If the latter is placed on the origin, then $X_{1156}=-X_{100}$, and $J'_{exc}$ passes through ${\pm}X_{100}$.

Let $F'$ and $J'_{exc}$ be copies of $F$ and $J_{exc}$ translated by $-X_{11}$ and $-X_{100}$ respectively, i.e., they become concentric with the EB (focal lengths are unchanged). Since their asymptotes are parallel to the EB axes and centered on the origin, their equations will be of the form:

\[
F': x\,y = k'_F,\;\;J'_{exc}: x\,y = k'_J
\]

\begin{remark}
$\lambda=2\sqrt{2k'_F}$, $\lambda'=2\sqrt{2k'_J}$, $\lambda'/\lambda=\sqrt{k'_J/k'_F}=\sqrt{2/\rho}$.
\end{remark} 

\begin{figure}
    \centering
    \includegraphics[width=\textwidth]{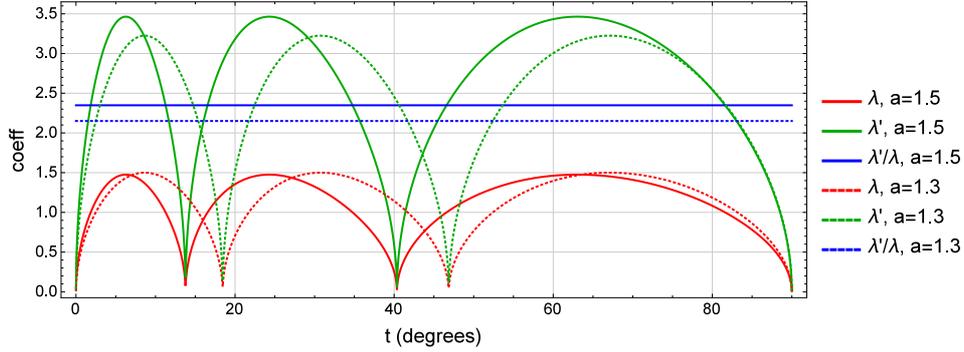}
    \caption{Focal lengths $\lambda,\lambda'$ of $F,J_{exc}$ vs the parameter $t$ in $P_1(t)=(a\cos{t},b\sin{t})$ are shown red and green. The solid (resp. dasheD) curves correspond to $a/b=1.5$ (resp. $a/b=1.3$). In the first quadrant there are 3 maxima. $\lambda'/\lambda$ (blue) remain constant for the whole interval.}
    \label{fig:focal-ratio}
\end{figure}

\begin{figure}
    \centering
    \includegraphics[width=\textwidth]{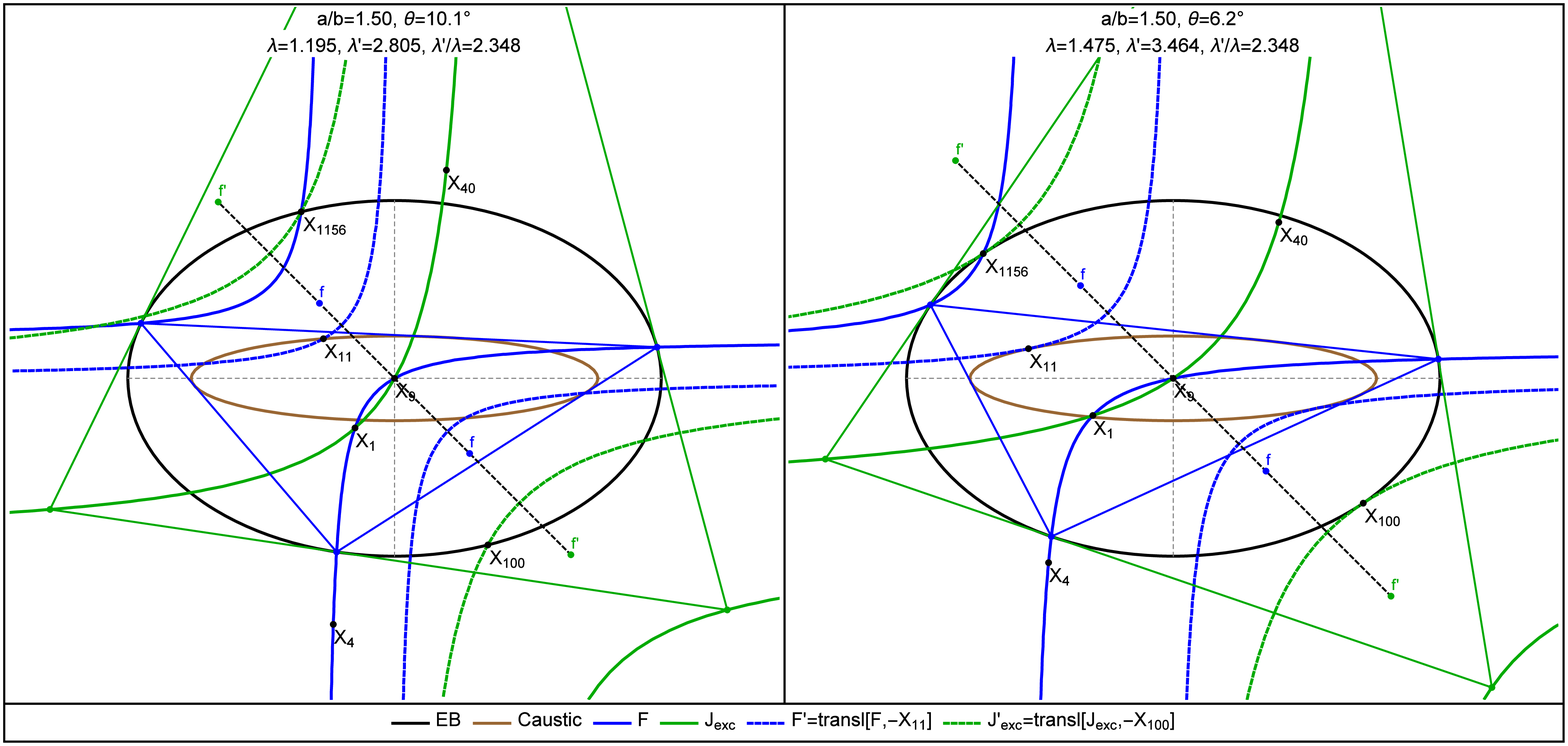}
    \caption{Two snapshots of $J$ and $F_{exc}$ drawn solid blue and solid green, respectively, for $a/b=1.5$. Also shown (dashed) are copies $F'$ and $J'_{exc}$ of both hyperbolas translated so they are dynamically concentric with the EB (translate $J$ by $-X_{11}$ and $F_{exc}$ by $-X_{100}$). Their focal lengths $\lambda,\lambda'$ are identical to the original ones; their focal axes are collinear and shown the dashed diagonal through the EB center. Notice that like $F$, $J'$ also intersects the EB at $X_{1156}$. \textbf{Left}: $t=10.1^\circ$, showing an intermediate value of ether focal length. \textbf{Right}: $t=6.2^\circ$, focal lengths are at a maximum. When this happens, the translated copy of $F$ (resp. $J_{exc}$) is tangent to the Caustic (resp. EB) at $X_{11}$ (resp. $X_{1156})$. \textbf{Video}: \cite[PL\#09,10]{reznik2020-playlist-circum}}
    \label{fig:zero-hyps}
\end{figure}

\


\section{Inconic Invariants}
\label{sec:inconic}
A triangle's {\em Inconic} touches its three sides while satisfying two other constraints, e.g., the location of its center. Similar to Circumconics, if the latter is interior to the 4 shaded regions in  Figure~\ref{fig:midlines} it is an ellipse, else it is a hyperbola. Lines drawn from each vertex to the Inconic tangency points concur at the perspector or Brianchon Point $B$ \cite{mw}.

Let the Inconic center $C$ be specified by Barycentrics\footnote{Barycentrics $g$ can be easily converted to Trilinears $f$ via: $f(s_1,s_2,s_3)=g(s_1,s_2,s_3)/s_1$ (cyclic) \cite{yiu2003}.} $g(s_1,s_2,s_3)$ (cyclic), then $B$ is given by $1/[g(s_2,s_3,s_1)+g(s_3,s_1,s_2)-g(s_1,s_2,s_3)]$ (cyclic) \cite{stothers2001-circumconics}. For example, consider the Inconic centered on $X_1$ (the Incircle), i.e., $g=s_2 s_3$ (cyclic). Then $B=1/(s_1 s_3+s_1 s_2-s_2 s_3)$. Dividing by the product ${s_1}{s_2}{s_3}$ obtain $B=1/(s_2 + s_3 - s_1)$ (cyclic), confirming that the perspector of the Incircle is the Gergonne Point $X_7$ \cite[Perspector]{mw}. The contact points are the vertices of the Cevian triangle through $B$.

Above we identified the confocal Caustic with the Mandart Inellipse $I_9$ \cite{mw} of the 3-periodic family, i.e., it is a stationary ellipse centered on $X_9$ and axis-aligned with the EB, Figure~\ref{fig:three-orbits-proof}. Its semi-axes $a_c,b_c$ are given by \cite{garcia2019-incenter}:

\begin{equation*}
a_c=\frac{a\left(\delta-{b}^{2}\right)}{a^2-b^2},\;\;\;
b_c=\frac{b\left({a}^{2}-\delta\right)}{a^2-b^2}\cdot
\end{equation*}

Similarly, the $X_9$-centered Inconic of the family of Excentral Triangles is the stationary EB, i.e., the EB is the Orthic Inconic \cite{mw} of the Excentrals.

\subsection{Excentral $X_3$-Centered Inconic}

No particular invariants have yet been found for any of the Inconics to 3-periodics with centers in $X_i$, $i=3,4,\ldots,12$. Let $I_3$ denote the $X_3$-centered inconic. Its Brianchon Point is $X_{69}$ and its foci aren't named centers \cite{moses2020-private-circumconic}.

\begin{remark}
$I_3$ is always an ellipse.
\end{remark}

To see this consider the Circumcenter of an acute, right, or obtuse triangle lies inside, on a vertex, or opposite to the interior of the Medial, respectively, i.e., within one of the 4 shaded regions in Figure~\ref{fig:midlines}.

Let $I'_3$ denote the $X_3$-centered Inconic of the Excentral Triangle. Its Brianchon Point is $X_{69}$ of the Excentral, i.e., $X_{2951}$ of the reference 3-periodic \cite{moses2020-private-circumconic}, Figure~\ref{fig:macbeath-excentral}(left). Let $\mu'_3,\mu_3$ denote $I_3'$ major and minor semi-axes. In \cite{reznik2020-poristic} we show:

\begin{proposition}
Let $T$ be a triangle and $T'$ its excentral. The Inconic $I_3'$ to $T'$ on its Circumcenter ($X_{40}$ of $T$) is a $90^\circ$-rotated copy of the $E_1$, the $X_1$-centered Circumconic, i.e.:

\begin{equation}
\mu_3'= \eta_1',\;\;\;\mu_3=\eta_1
\end{equation}
\end{proposition}

\begin{figure}
    \centering    \includegraphics[width=\textwidth]{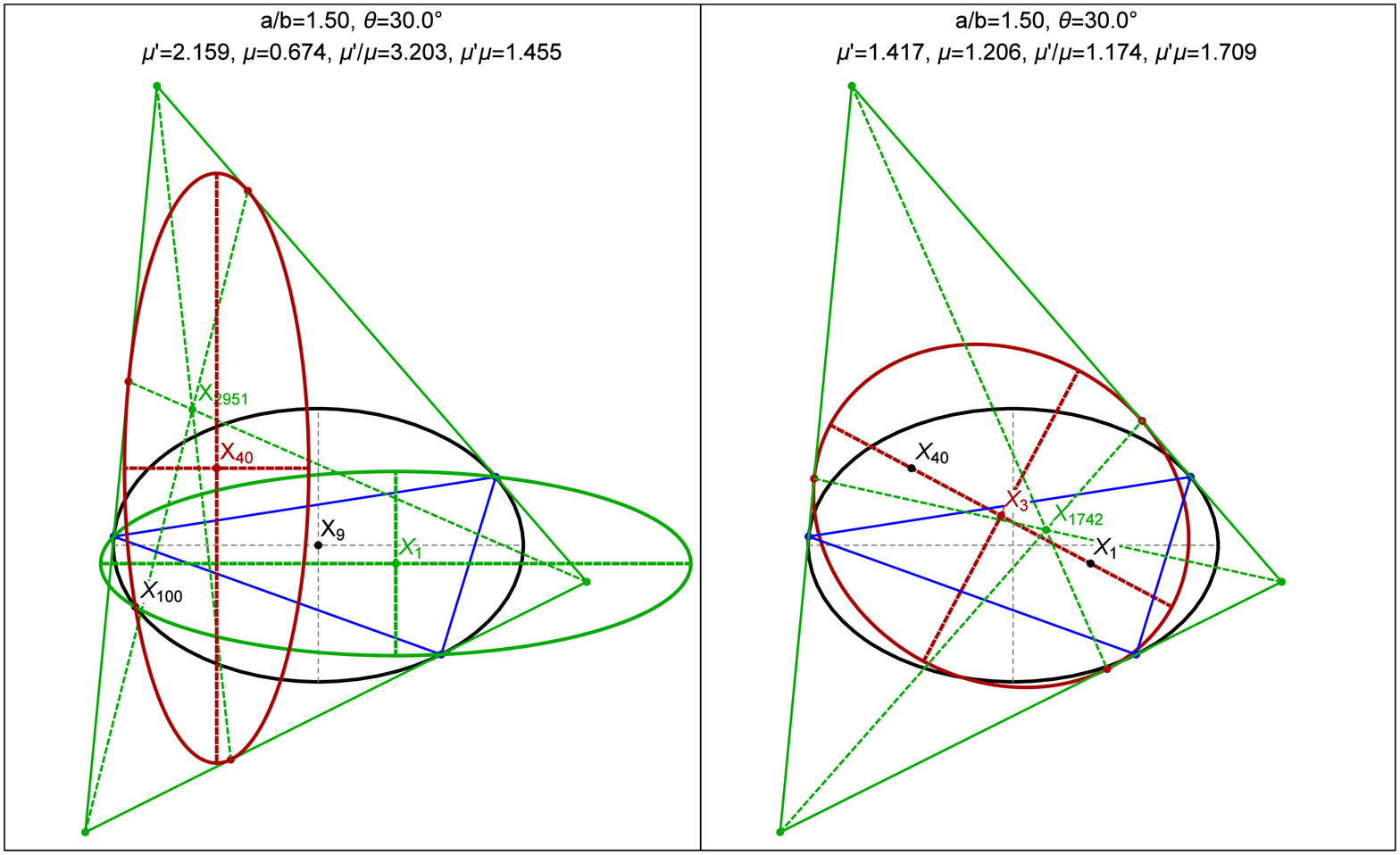}
    \caption{\textbf{Left}: The $X_{3}$-centered Inconic $I_3'$ (red) of the Excentral Triangle (green), whose center is $X_{40}$ in terms of the 3-periodic (blue). Its Brianchon Point is $X_{2951}$ \cite{moses2020-private-circumconic}. Its axes are aligned with the EB which it intersects at $X_{100}$, and its aspect ratio is invariant. The $X_1$-centered Circumconic $E_1$ (green) is a $90^\circ$-rotated copy of $I_3'$. \textbf{Right}: The MacBeath Inconic $I_5'$ of the Excentral (red) is centered on the latter's $X_5$ and has foci on $X_4$ and $X_3$. These correspond to  $X_3$, $X_1$, and $X_{40}$ of the reference triangle (blue). Its Brianchon Point is $X_{1742}$ \cite{moses2020-private-circumconic}. Though its axes are askew with respect to the EB, its aspect ratio is invariant over the 3-periodic family. \textbf{Video}: \cite[PL\#11]{reznik2020-playlist-circum}.}
    \label{fig:macbeath-excentral}
\end{figure}

\begin{remark}
As $E_1$, $I'_3$ is axis-aligned with the EB and contains $X_{100}$, Figure~\ref{fig:macbeath-excentral} (left).
\end{remark}

\begin{theorem}
Over 3-periodics, $\mu_3'$ and $\mu_3$ are variable though their ratio is invariant. These are given by:

\begin{align*}
\mu_3'=&R+d,\;\;\;\mu_3=R-d\\
\frac{\mu_3'}{\mu_3}=&
\frac{R+d}{R-d} = \frac{1+\sqrt{1-2\rho}}{\rho}-1>1
\end{align*}
\label{thm:excIncX3}
\end{theorem}

\noindent As above, $d=\sqrt{R(R-2r)}$, and $\rho=r/R$.

\begin{proof}
The above was initially a candidate an isosceles 3-periodic and then verified it held any 3-periodic configuration with a CAS. A related proof appears in \cite{reznik2020-poristic}.
\end{proof}

\subsection{Excentral  X$_5$-Centered (MacBeath) Inconic}

Above the locus of the Excenters is identified with the Excentral MacBeath Circumconic $E_6' $. The MacBeath {\em Inconic} $I_5$ of a triangle is centered on $X_5$ and has foci on $X_4$ and $X_3$. Its Brianchon Point is $X_{264}$, and it can be both an ellipse or a hyperbola \cite[MacBeath Inconic]{mw}. No invariants have yet been found for $I_5$.

Consider $I'_5$, the MacBeath Inconic of the Excentral Triangle, Figure~\ref{fig:macbeath-excentral}(right). The center and foci of $I'_5$ with respect to the reference triangle are $X_3$, $X_1$, and $X_{40}$, respectively, and its Brianchon is $X_{1742}$ \cite{moses2020-private-circumconic}. Unlike $I'_3$, the axes of $I'_5$ are askew with respect to the EB.

\begin{remark}
$I'_5$ is always an ellipse.
\end{remark}

This is due to the fact that the Excentral is acute as is its homothetic Medial. Since $X_5$ is the latter's Circumcenter, it must lie inside it.

\noindent Let $\mu'_5,\mu_5$ denote $I_5'$ major and minor semi-axes.

\begin{theorem}
Over the 3-periodic family $\mu_5'$ and $\mu_5$ are variable and their ratio is invariant. These are given by:

\begin{align*}
\mu_5'=&R,\;\;\;\mu_5=\sqrt{R^2-d^2}\\
 \frac{\mu_5'}{\mu_5}=&\frac{R}{\sqrt{R^2-d^2}}=\frac{1}{\sqrt{2 \rho}}\\
\end{align*}
\label{thm:excIncX5}
\end{theorem}

The above was derived for an isosceles 3-periodic and shown to work for any $a,b$. A related proof is given in \cite{reznik2020-poristic}.

\section{Conclusion}
\label{sec:conclusion}
The various quantities and invariants described above are summarized on Table~\ref{tab:invariants}.

\begin{table}[H]
\begin{tabular}{|c|l|l|l|}
\hline
Length & Meaning & Value & Invariant \\
\hline
$\eta_1'$ & Major semiaxis of $E_1$ & $R+d$ & -- \\
$\eta_1$ & Minor semiaxis of $E_1$ & $R-d$ & -- \\
$\eta_1'/\eta_1$ & $E_1$ aspect ratio & 
$(1-\rho+\sqrt{1-2\rho})/\rho$ & yes \\
\hline
$\lambda'$ & $J_{exc}$ focal length & -- & -- \\
$\lambda$ & $F$ focal length & -- & -- \\
$\lambda'/\lambda$ & Ratio of focal lengths & $\sqrt{2/\rho}$ & yes \\
\hline
$\mu_3'$ & Major semiaxis of $I_3'$ & $=\eta_1'$ & -- \\
$\mu_3$ & Minor semiaxis of $I_3'$ & $=\eta_1$ & -- \\
$\mu_3'/\mu_3$ & $I_3'$ aspect ratio &  $=\eta_1'/\eta_1$ & yes \\
\hline
$\mu_5'$ & Major semiaxis of $I_5'$ & $R$ & -- \\
$\mu_5$ & Minor semiaxis of $I_5'$ & $\sqrt{R^2-d^2}$ & -- \\
$\mu_5'/\mu_5$ & $I_5'$ aspect ratio & ${1}/{\sqrt{2\rho}}$ &  yes \\
\hline
\end{tabular}
\caption{Summary of Circumconic and Inconic Quantities. $\rho=r/R$ and $d=\sqrt{R(R-2r)}$.}
\label{tab:invariants}
\end{table}

Videos mentioned above have been placed on a \href{https://bit.ly/379mk1I}{playlist} \cite{reznik2020-playlist-circum}. Table~\ref{tab:playlist} contains quick-reference links to all videos mentioned, with column ``PL\#'' providing video number within the playlist.

Additionally to Conjecture~\ref{conj:moses} we submit the following questions to the reader:

\begin{itemize}
  \item Can expressions be derived for $\lambda'$ and $\lambda$ (we only provide one for their ratio)?
  \item Can alternate proofs be found for Theorems~\ref{thm:axis-ratio} and \ref{thm:focal-ratio} with tools of algebraic and/or projective geometry?
  \item Are there other notable circumconic pairs which exhibit interesting invariants?
  \item Can any of the invariants cited above be generalized to $N$-periodics?
\item Are there other Inconic invariants over 3-periodics and/or their derived triangles? 
    \item Are there interesting properties for the loci of the {\em foci} of Feuerbach and/or Jerabek Circumhyperbolas?
    \item The Yff Circumparabola whose focus is on $X_{190}$ is shown in \cite[PL\#12]{reznik2020-playlist-circum} over 3-periodics. Are there interesting invariants?
    \item Does any Triangle Circumcubic display interesting invariants over 3-periodics? We found none for the Thomson Cubic shown in \cite[PL\#13]{reznik2020-playlist-circum}. How about the Darboux, Neuberg, Lucas, and myriad others catalogued in \cite{gibert2020-cubics}.
\end{itemize}

\begin{table}
\begin{tabular}{c|l|l}
\href{https://bit.ly/2NOOIOX}{PL\#} & Title & Section\\
\hline

\href{https://youtu.be/tMrBqfRBYik}{01} &
{Mittenpunkt stationary at EB center} & \ref{sec:intro} \\

\href{https://youtu.be/yEu2aPiJwQo}{07} & \makecell[lt]{Invariant Aspect Ratio of \\Circumbilliard of Poristic Family} &
\ref{sec:intro} \\

\href{https://youtu.be/P_Io7HsWGnQ}{08} &
{The $X_1$- and $X_2$-centered Circumellipses} &
\ref{sec:circumellipses} \\

\href{https://youtu.be/Pz4tUijYZCA}{09} &
\makecell[lt]{Orbit Feuerbach and \\ Excentral Jerabek Circumhyperbolas} &
\ref{sec:circumhyperbolae} \\


\href{https://youtu.be/ewioM6-nCpY}{10} &
Invariant Focal Length Ratio for $F$ and $J_{exc}$ &
\ref{sec:circumhyperbolae} \\

\href{https://youtu.be/CHbrZvx1I8w}{11} &
\makecell[lt]{
Excentral MacBeath and $X_3$-Centered\\Inconics: Invariant Aspect Ratio
} &
\ref{sec:inconic} \\

\href{https://youtu.be/Sm9g5lqhZbk}{12} & The Yff Circumparabola of 3-Periodics &
\ref{sec:conclusion} \\

\href{https://youtu.be/s-h72iorZKw}{13} & The Thomson Cubic of 3-Periodics &
\ref{sec:conclusion} \\


\end{tabular}
\caption{Videos mentioned in the paper. Column ``PL\#'' indicates the entry within the playlist \cite{reznik2020-playlist-circum}}
\label{tab:playlist}
\end{table}

\section*{Acknowledgments}
We would like to thank Peter Moses and Clark Kimberling, for their prompt help with dozens of questions. We would like to thank Boris Odehnal for his help with some proofs. A warm thanks goes out to Profs. Jair Koiller and Daniel Jaud who provided critical editorial help.

The second author is fellow of CNPq and coordinator of Project PRONEX/ CNPq/ FAPEG 2017 10 26 7000 508.

\bibliographystyle{spmpsci}
\bibliography{elliptic_billiards_v3,authors_rgk_v1} 

\appendix
\section{Computing a  Circumconic}
\label{app:circum-linear}
Let a Circumconic have center $M=(x_m,y_m)$ Equation~\eqref{eqn:e0} is subject to the following 5 constraints\footnote{If $M$ is set to $X_9$ one obtains the Circumbilliard.}: it must be satisfied for vertices $P_1,P_2,P_3$, and its gradient must vanish at $M$:

\begin{align*}
f(P_i)=&\;0,\;\;\;i=1,2,3\\
\frac{dg}{dx}(x_m,y_m)=&\;c_1+c_3 y_m+2c_4 x_m=0\\
\frac{dg}{dy}(x_m,y_m)=\;&c_2+c_3 x_m+2c_5 y_m=0
\end{align*}

Written as a linear system:

$$
\left[
\begin{array}{ccccc}
x_1&y_1&x_1 y_1&x_1^2&y_1^2\\
x_2&y_2&x_2 y_2&x_2^2&y_2^2\\
x_3&y_3&x_3 y_3&x_3^2&y_3^2\\
1&0&y_m&2\,x_m&0\\
0&1&x_m&0&2\,y_m
\end{array}
\right] .
\left[\begin{array}{c}c_1\\c_2\\c_3\\c_4\\c_5\end{array}\right] =
\left[\begin{array}{c}-1\\-1\\-1\\0\\0\end{array}\right]
$$

Given sidelenghts $s_1,s_2,s_3$, the coordinates of $X_9=(x_m,y_m)$ can be obtained by converting its Trilinears $\left(s_2 + s_3 - s_1 :: ...\right)$ to Cartesians \cite{etc}. 

Principal axes' directions are given by the eigenvectors of the Hessian matrix $H$ (the jacobian of the gradient), whose entries only depend on $c_3$, $c_4$, and $c_5$:

\begin{equation}
H = J(\nabla{g})=\left[\begin{array}{cc}2\,c_4&c_3\\c_3&2\,c_5\end{array}\right]
\label{eqn:hessian}
\end{equation}

The ratio of semiaxes' lengths is given by the square root of the ratio of $H$'s eigenvalues:

\begin{equation}
a/b=\sqrt{\lambda_2/\lambda_1}
\label{eqn:ratiolambda}
\end{equation}

Let $U=(x_u,y_u)$ be an eigenvector of $H$. The length of the semiaxis along $u$ is given by the distance $t$ which satisfies:

$$
g(M + t\,U) = 0
$$

This yields a two-parameter quadratic $d_0 + d_2 t^2$, where:

$$
\begin{array}{cll}
d_0 & = & 1 + c_1 x_m + c_4 x_m^2 + c_2 y_m + c_3 x_m y_m + c_5 y_m^2 \\ 
d_2 & = & c_4 x_u^2 + c_3 x_u y_u + c_5 y_u^2
\end{array}
$$

The length of the semi-axis associated with $U$ is then $t=\sqrt{-d_0/d_2}$. The other axis can be computed via \eqref{eqn:ratiolambda}.

The eigenvectors (axes of the conic) of $H$ are given by the  zeros of the quadratic form
\begin{align*}
   q(x,y)= c_3(y^2-x^2)+2(c_2-c_5)xy
\end{align*}

\section{Circumellipses of Elementary Triangle}
\label{app:circum-x1x2x9}
Let a triangle $T$ have vertices $P_1=(0,0)$, $P_2=(1,0)$ and $P_3=(u,v)$ and sidelengths $s_1,s_2,s_3$. Using the linear system in Appendix~\ref{app:circum-linear}, one can obtain implicit equations for the circumellipses $E_9,E_1,E_2$ centered on $T$'s Mittenpunkt $X_9$, Incenter $X_1$, and Barycenter $X_2$, respectively:

\begin{align*}
E_9(x,y)=&  v^2 x^2-v(s_1 -s_2-1+2u)xy+((s_1-s_2-1)u+u^2+s_2 )y^2\\
-&v^2x+((s_1-s_2-1)u+u^2+s_2)y^2+v(u-s_2)y=0\\
E_1(x,y)=&  \left( L-2 \right) {v}^{2}{x}^{2} + \left( L-2\,{ s_2}-2\,u
 \right)   \left( L-2 \right) v\, xy \\
 +& \left( 
-{L}^{2}u+ \left( 2\,u+1 \right) L{ s_2}+ \left( {u}^{2}+2\,u
 \right) L-2\,{s_2}^{2}-4\,u{ s_2}-2\,{u}^{2} \right) {y}^{2}\\
 -& \left( L-2 \right) {v}^{2} x
 - v\left( L{ s_2}-uL-2\,{s_2}^{2}+2\,u \right) y 
\\
E_2(x,y)=&v^2x^2+v(1-2u)xy+(u^2-u+1)y^2-v^2x+v(u-1)y=0\\
s_1=& \sqrt{(u-1)^2+v^2},\;\;\; s_2=\sqrt{u^2+v^2}, \;\;\; L=s_1+s_2+1
\end{align*}

Consider the quadratic forms
\begin{align*}
    q_9(x,y)=&   v\left(  { s_1}\, -{ s_2}\, +2 u  -1 \right) {x}^{2}+ 2\left(  (s_2-s_1)u
 - { s_2} - \,{u}^{2}+ u+  {v}^{2}
 \right) xy\\
 +& v\left( 1-2u - { s_1}\, +{ s_2}\,    \right) {y}^{2}
 \\
 q_1(x,y)=& -v \left( L-2 \right)  \left( L-2\,{ s_2}-2\,u \right) {x}^{2}+v  \left( L-2 \right)  \left( L-2\,{ s_2}-2\,u\right)y^2 \\
 +&
2 (  L^2u-(2u+1)Ls_2+(v^2-u^2-2u)L+ 4s_2u+4u^2 ) xy\\
q_2(x,y)=&v(2u+1)x^2+2(-u^2+v^2+u-1)xy+v(1-2u)y^2
\end{align*}

The axes of $E_9$ (resp. $E_1$) are defined by the zeros of $q_9$ (resp. $q_1$). Using the above equations it is straightforward to show that the axes of $E_1$ and $E_9$ are parallel.

The axes of $E_2$ and $E_9$ are parallel if and only if $(u-1)^2+v^2=1$ or $u^2+v^2=1$; this means that the triangle is isosceles.

The implicit equations of the circumhyperbolas $F $ passing through the vertices of the orbit centered on $X_{11}$ and   $J_{exc} $ passing through the vertices of the excentral triangle and centered on $X_{100}$ are:
{\small 
\begin{align*}
F(x,y)=& v^3(2 u  -1) (x^2-y^2)+  v^3(1-2u)x\\
+&[( s_2^3+(u-1) s_2^2-u s_2) s_1+(2u-1)s_2^2-us_1^2s_2 -u^4-4 u^2 v^2+v^4+2 uv^2] x y\\
 +&[ s_1^2 u^2  s_2+(-u  s_2^3-u (u-1)  s_2^2+ s_2 u^2)  s_1+u  s_2^4-v^2  s_2^2-u^3 (2 u-1)]y=0\\
 J_{exc}(x,y)=&  4v^3\left( 2\,u -1\, \right) ({x}^{2}-y^2)\\
 +& [ \left( 4\,s_2^{3}+ 4\left(  u-1
 \right)     s_2 ^{2}-4\,u{   s_2} \right) {   s_1}  -4\,u 
   s_1 ^{2}{   s_2} -4\,    s_2 ^{2} \\
   -&4\,{u}^{4}-16\,{
u}^{2}{v}^{2}+4\,{v}^{4}+8\,{u}^{3}+16\,u{v}^{2}
 ] xy\\
 +&[ (2 (- s_1  s_2^3+( (1- s_1) u-v^2+ s_1)  s_2^2+u  s_1 ( s_1+1)  s_2+u (u-2) s_1^2)) v]x\\
 +&[ (4-4 u) s_2  ((u^2-( \frac{1}{2} s_1 +1) u+ \frac{1}{2} s_1 +\frac{1}{2}) s_2 + \frac{1}{2} u s_1  (s_1 +1)- \frac{1}{2} s_2 ^2 (s_2 +s_1 ))]y\\
   -& u  s_1^2 v  s_2+(v  s_2^3+(-1+u) v  s_2^2-u v  s_2)  s_1+v  s_2^4-(2 u^2-2 u+1) v  s_2^2=0
\end{align*}
}
 Using the above equations it is straightforward to show that the axes of $E_9$ and asymptotes  of $F $ and $J_{exc} $ are parallel.

\section{Circumellipses with Parallel Axes}
\label{app:ce_parallel}
Consider a triangle with vertices $A=(u,v)$, $B=(-1,0)$ and $C=(1,0).$ Let $s_2=|A-B|$ and $s_3=|B-C|$.
The equation of $F_{med}$ is given by:
{\small  
\begin{equation}    
\aligned 
F_{med}(x,y)=&4 u v^3(x^2- y^2) (2( u-1)  s_3  s_2^2-( 2(u-1)  s_3^2-2( u^2+ v^2-1) s_3)s_2\\ 
&2(u^2-1)^2+2v^2(4u^2-v^2)) xy\\
&+(-((u^2+v^2-1)  s_3+ s_3^2 (u-1)) v  s_2+ s_2^2 (u+1) v  s_3\\
-& v(v^4+(u^2-1)^2 ) x\\
+&(2 u^2+2 v^2-2)   u v^2\; y=0
\endaligned
\end{equation}
}
A one parameter family of circumellipses passing through $A$, $B$, $C$ and $X_{100}$ is given by:

\begin{equation}\aligned  
 E_b(x,y)
 =&  u v^3(bxy  -4  x^2)\\
 + &((b (u^2+v^2-1)  s_3+b (u-1)  s_3^2)  s_2-b (u+1)  s_2^2  s_3\\
 -&v^2(b(4u^2-v^2)+4uv)-b(u^2-1)^2) y^2\\ 
 -&v ((b (u^2+v^2-1)  s_2-b (u+1)  s_2^2)  s_3+b (u-1)  s_3^2  s_2\\
-&b(u^2+v^2-1) (  u^2-  v^2+  u v-1)) y+ 4 u v^3=0
 \endaligned
\end{equation}

It is
straightforward to verify that the family $E_b$ is
axis aligned (independent of b). Denoting the center of $E_b$ by $(x_c,y_c)$ it follows, using CAS, that $F_{med}(x_c,y_c)=0.$

The reciprocal
follows similarly.

\section{Table of Symbols}
\label{app:symbols}
Tables~\ref{tab:kimberling} and \ref{tab:symbols} lists most Triangle Centers and symbols mentioned in the paper.

\begin{table}[H]
\scriptsize
\begin{tabular}{|c|l|l|}
\hline
Center & Meaning & Note\\
\hline
$X_1$ & Incenter & Locus is Ellipse \\
$X_2$ & Barycenter & Perspector of Steiner Circum/Inellipses \\
$X_3$ & Circumcenter & Locus is Ellipse, Perspector of $M$ \\
$X_4$ & Orthocenter & \makecell[tl]{Exterior to EB\\ when 3-periodic is obtuse} \\
$X_5$ & Center of the 9-Point Circle & \\
$X_6$ & Symmedian Point & Locus is Quartic \cite{garcia2020-ellipses} \\
$X_7$ & Gergonne Point & Perspector of Incircle \\
$X_8$ & Nagel Point & Perspector of $I_9$, $X_1$ of ACT Incircle \\
$X_9$ & Mittenpunkt & Center of (Circum)billiard \\
$X_{10}$ & Spieker Point & Incenter of Medial \\
$X_{11}$ & Feuerbach Point & on confocal Caustic \\
$X_{40}$ & Bevan Point & $X_3$ of Excentral \\
$X_{69}$ & $X_6$ of the ACT & Perspector of $I_3$ \\
$X_{100}$ & Anticomplement of $X_{11}$ & On Circumcircle and EB, $J_{exc}$ center \\
$X_{125}$ & Center of Jerabek Hyperbola $J$ & \\
$X_{142}$ & $X_9$ of Medial & Midpoint of $X_9{X_{7}}$, lies on $L(2,7)$ \\ 
$X_{144}$ & Anticomplement of $X_7$ &\makecell[tl]{Perspector of ACT\\and its Intouch Triangle} \\
$X_{168}$ & $X_9$ of the Excentral Triangle &  Non-elliptic Locus\\
$X_{190}$ & Focus of the Yff Parabola & Intersection of $E_2$ and the EB \\
$X_{264}$ & Isotomic Conjugate of $X_3$ & Perspector of $I_5$ \\
$X_{649}$ & Cross-difference of $X_1,X_2$ & Perspector of $J_{exc}$ \\ 
$X_{664}$ &  Trilinear Pole of $L( 2,7)$ & Intersection of $E_1$ and $E_2$ \cite{moses2020-private-circumconic} \\
$X_{650}$ & Cross-difference of $X_1,X_3$ & Perspector of $F$ \\
$X_{1156}$ & \makecell[tl]{Isogonal Conjugate of\\Schröder Point $X_{1155}$} &  intesection of $F$ with EB \\
$X_{1742}$ & Mimosa Transform of $X_{212}$ & Perspector of $I_5'$ \\
$X_{2951}$ & Excentral-Isogonal Conjugate of $X_{57}$ & Perspector of $I_3'$ \\
$X_{3035}$ & Complement of $X_{11}$ & Center of $F_{med}$ \\
$X_{3659}$ & $X_{11}$ of Excentral Triangle & Center of $F_{exc}$ \\
\hline
$L(2,7)$ & ACT-Medial Mittenpunkt Axis & Line $\mathcal{L}_{663}$  \cite{etc_central_lines} \\
$L(1,3)$ & Isogonal Conjugate of $F$ and $J_{exc}$ &  Line $\mathcal{L}_{650}$ \cite{etc_central_lines}  \\
\hline
\end{tabular}
\caption{Kimberling Centers and Central Lines mentioned in the paper.}
\label{tab:kimberling}
\end{table}

\begin{table}
\scriptsize
\begin{tabular}{|c|l|l|}
\hline
Symbol & Meaning & Note\\
\hline
$a,b$ & EB semi-axes & $a>b>0$\\ 
$P_i,s_i$ & Vertices and sidelengths of 3-periodic & invariant $\sum{s_i}$ \\
$P_i'$ & Vertices of the Excentral Triangle & \\
$a_c,b_c$ & Semi-axes of confocal Caustic & \\ 
$r,R,\rho$ & Inradius, Circumradius, $r/R$ & $\rho$ is invariant \\
$\delta$ & Oft-used constant & $\sqrt{a^4-a^2 b^2+b^4}$ \\
$d$ & Distance $|X_1{X_3}|$ & $\sqrt{R(R-2r)}$\\
$P^\perp$ & Obtuse 3-periodic limits on EB & \\
\hline
$F,J$ & Feuerbach, Jerabek Hyperbola & Centers $X_{11},X_{125}$, Perspectors $X_{650},X_{647}$ \\
$F_{exc}$ & $F$ of Excentral Triangle & Center $X_{3659}$ \cite{moses2020-private-circumconic} \\
$J_{exc}$ & $J$ of Excentral Triangle & Center $X_{100}$, Perspector $X_{649}$ \\
$F',J_{exc}'$ & $F,J_{exc}$ translated by $-X_{11},-X_{100}$ & Origin-centered \\
$F_{med}$ & $F$ of Medial & Center $X_{3035}$ 
\cite{moses2020-private-circumconic} \\
\hline
$E_i$ & Circumellipse centered on $X_i$ & \makecell[tl]{Axes parallel to $E_9$ if $X_i$ on $F_{med}$} \\
$E_6'$ & Excentral MacBeath Circumellipse & Center $X_9$, Perspector $X_{40}$ \\
\hline
$I_3,I_5$ & Inellipses Centered on $X_3$,$X_5$ & Perspectors $X_{69},X_{264}$ \\
$I_9$ & Mandart Inellipse & Perspector $X_8$ \\
$I'_3$ & Excentral $I_3$ & Center $X_{40}$, Perspector $X_{2951}$ \cite{moses2020-private-circumconic} \\
$I'_5$ & Excentral $I_5$ & Center $X_{3}$, Perspector $X_{1742}$ \cite{moses2020-private-circumconic}\\
\hline
$\eta'_i,\eta_i$ & Major and minor semiaxis of $E_i$ & Invariant ratio if $X_i$ on $F_{med}$\\
$\mu_i',\mu_i$ & Major and minor semiaxis of $I_i'$ & Invariant ratio for $i=3,5$ \\
$\lambda',\lambda$ & Focal lengths of $J_{exc},F$ (and $J_{exc}',F'$) & Invariant ratio \\
\hline
\end{tabular}
\caption{Symbols mentioned in the paper.}
\label{tab:symbols}
\end{table}

\end{document}